\documentclass[a4paper,11pt]{amsart}

\usepackage[french, english]{babel} 
\usepackage[T1]{fontenc}
\usepackage[ansinew]{inputenc}

\date{}
\title[Semigroups generated by fractional Ornstein-Uhlenbeck operators]{Description of the smoothing effects of semigroups generated by fractional Ornstein-Uhlenbeck operators and subelliptic estimates}

\author{Paul Alphonse}
\address{Univ Rennes, CNRS, IRMAR - UMR 6625, F-35000 Rennes, France}
\email{paul.alphonse@ens-lyon.fr}

\keywords{fractional Ornstein-Uhlenbeck operators, Gevrey-type regularity, Kalman rank condition, subelliptic estimates}

\subjclass[2010]{47D06, 35B65, 35H20}

\usepackage[top=3cm, bottom=2cm, left=3cm, right=3cm]{geometry}		

\usepackage{enumitem}

\usepackage{array}											

\usepackage{amsfonts}
\usepackage{amsmath}
\usepackage{amsthm}
\usepackage{amssymb}

\usepackage{bbm}

\usepackage{stmaryrd}

\usepackage[scr]{rsfso}

\usepackage{comment}

\usepackage[dvipsnames]{xcolor}

\usepackage{hyperref}

\hypersetup{	
colorlinks=true,
breaklinks=true,
urlcolor= red,
linkcolor= red,
citecolor=ForestGreen
}

\numberwithin{equation}{section}
										
\newtheorem{thm}{Theorem}[section]
\newtheorem{prop}[thm]{Proposition}
\newtheorem{lem}[thm]{Lemma}
\newtheorem{cor}[thm]{Corollary}
\theoremstyle{definition}

\DeclareMathOperator{\Tr}{Tr}
\DeclareMathOperator{\Ker}{Ker}
\DeclareMathOperator{\Ran}{Ran}

\DeclareMathOperator{\Rank}{Rank}

\DeclareMathOperator{\Span}{Span}

\begin{document}

\sloppy

\selectlanguage{english}

\begin{abstract}
We study semigroups generated by general fractional Ornstein-Uhlenbeck operators acting on $L^2(\mathbb R^n)$. We characterize geometrically the partial Gevrey-type smoothing properties of these semigroups and we sharply describe the blow-up of the associated seminorms for short times, generalizing the hypoelliptic and the quadratic cases. As a byproduct of this study, we establish partial subelliptic estimates enjoyed by fractional Ornstein-Uhlenbeck operators on the whole space by using interpolation theory.
\end{abstract}

\maketitle

\section{Introduction}

We study semigroups generated by fractional Ornstein-Uhlenbeck operators. These are non-local and non-selfadjoint operators in general, sum of a fractional diffusion and a linear transport operator. More precisely, the fractional Ornstein-Uhlenbeck operator associated with the positive real number $s>0$ and the real $n\times n$ matrices $B$ and $Q$, with $Q$ symmetric positive semidefinite, is given by
\begin{equation}\label{22072019E2}
	\mathcal P = \frac12\Tr^s(-Q\nabla^2_x) + \langle Bx,\nabla_x\rangle,\quad x\in\mathbb R^n.
\end{equation}
In this definition, the operator $\Tr^s(-Q\nabla^2_x)$ stands for the Fourier multiplier associated with the symbol $\langle Q\xi,\xi\rangle^s$, and $\langle Bx,\nabla_x\rangle$ denotes the following differential operator
$$\langle Bx,\nabla_x\rangle = \sum_{i=1}^n\sum_{j=1}^nB_{i,j}x_j\partial_{x_i},\quad B = (B_{i,j})_{1\le i,j\le n}.$$

The operator $\mathcal P$ equipped with the domain 
\begin{equation}\label{17062020E2}
	D(\mathcal P) = \big\{u\in L^2(\mathbb R^n) : \mathcal Pu\in L^2(\mathbb R^n)\big\},
\end{equation}
is known from the work \cite{AB} (Theorem 1.1) to generate a strongly continuous semigroup $(e^{-t\mathcal P})_{t\geq0}$ on $L^2(\mathbb R^n)$. The purpose of this work is to understand how the possible non-commutation phenomena between the diffusion and the transport parts of the operator $\mathcal P$ allow the associated evolution operators $e^{-t\mathcal P}$ to enjoy smoothing properties in specific directions of the space we aim at sharply and completely describing. More precisely, our objective is to exhibit the vector spaces $\Sigma\subset\mathbb R^n$ containing the smoothing directions of the operators $e^{-t\mathcal P}$, i.e. such that for all $t>0$, $m\geq1$ and $\xi_1,\ldots,\xi_m\in\Sigma$, there exists a positive constant $c_{t,m,\xi_1,\ldots,\xi_m}>0$ such that for all $u\in L^2(\mathbb R^n)$, 
\begin{equation}\label{17062020E1}
	\big\Vert\langle\xi_1,\nabla_x\rangle\ldots\langle\xi_m,\nabla_x\rangle e^{-t\mathcal P}u\big\Vert_{L^2(\mathbb R^n)}\le c_{t,m,\xi_1,\ldots,\xi_m}\ \Vert u\Vert_{L^2(\mathbb R^n)},
\end{equation}
and to sharply describe the dependence of the constant $c_{t,m,\xi_1,\ldots,\xi_m}$ with respect to $m\geq1$ and $\xi_1,\ldots,\xi_m\in\Sigma$, at least for short times $0<t\ll1$. This kind of estimates have already been obtained with a precise control of the short times asymptotics in the case $s=1$ in the work \cite{AB2} (Example 2.7). On the other hand, the study of the smoothing properties of semigroups generated by fractional Ornstein-Uhlenbeck operators has already been performed in the work \cite{AB} (Theorem 1.2) in a hypoelliptic setting, i.e. when the matrices $B$ and $Q$ satisfy an algebraic condition known as the Kalman rank condition and defined just after \eqref{22072019E1}. The present work aims at unifying these two results in a general setting for the matrices $B$ and $Q$, and for general positive real numbers $s>0$. An application to the study of the partial subelliptic estimates enjoyed by the operator $P$ on the whole space will also be given, generalizing results from the two same works \cite{AB2} (Example 2.11) and \cite{AB} (Corollary 1.15).

The Ornstein-Uhlenbeck operators (case $s=1$) and their associated semigroups have been very much studied in the last two decades. The structure of theses operators was analyzed in \cite{MR1289901}, while their spectral properties were investigated in \cite{MR1941990,  MR3342487}. The smoothing properties of the associated semigroups were studied in \cite{MR2505366, MR2257846, MR3710672,MR1389786,MR1941990, MR3342487} and some global hypoelliptic estimates were derived in \cite{MR2729292, MR2257846, MR3710672, MR3342487}. We also refer the reader to \cite{MR1343161, MR1475774} where the Ornstein-Uhlenbeck operators are studied while acting on spaces of continuous functions. We recall from these works that among the different characterizations of the hypoellipticity of these operators is the Kalman rank condition
\begin{equation}\label{22072019E1}
	\Rank\big[B\ \vert\ \sqrt Q\big] = n,
\end{equation}
where the $n\times n^2$ matrix 
$$\big[B\ \vert\ \sqrt Q\big] = \big[\sqrt Q,B\sqrt Q,\ldots,B^{n-1}\sqrt Q\big],$$
is obtained by writing consecutively the columns of the matrices $\sqrt Q, B\sqrt Q,\ldots,B^{n-1}\sqrt Q$, in addition to different assertions derived from the classical H\"ormander's work on the hypoellipticity of differential operators, see e.g. the introduction of \cite{AB}. The Kalman rank condition will naturally appear in the following.

\subsubsection*{Outline of the work} In Section \ref{results}, we present the main results contained in this work. Section \ref{reg} is devoted to the study of the smoothing properties of semigroups generated by fractional Ornstein-Uhlenbeck operators. Section \ref{Appendix} is an Appendix containing the proofs of some technical results.

\subsubsection*{Notations} The following notations and conventions will be used all over the work:
\begin{enumerate}[label=\textbf{\arabic*.},leftmargin=* ,parsep=2pt,itemsep=0pt,topsep=2pt]
\item The canonical Euclidean scalar product of $\mathbb R^n$ is denoted by $\langle\cdot,\cdot\rangle$ and $\vert\cdot\vert$ stands for the associated canonical Euclidean norm. The orthogonality with respect to this scalar product, as for it, is denoted $F^{\perp}$ for all $F\subset\mathbb R^n$.
\item The inner product of $L^2(\mathbb R^n)$ is defined for all $u,v\in L^2(\mathbb R^n)$ by
$$\langle u,v\rangle_{L^2} = \int_{\mathbb R^n}u(x)\overline{v(x)}\ \mathrm dx,$$
while $\Vert\cdot\Vert_{L^2}$ stands for the associated norm.
\item For all function $u\in\mathscr S(\mathbb R^n)$, the Fourier transform of $u$ is denoted by $\widehat u$ and defined by
$$\widehat u(\xi) = \int_{\mathbb R^n}e^{-i\langle x,\xi\rangle}u(x)\ \mathrm dx.$$
With this convention, Plancherel's theorem states that 
$$\forall u\in\mathscr S(\mathbb R^n),\quad \Vert\widehat u\Vert^2_{L^2} = (2\pi)^n\Vert u\Vert^2_{L^2}.$$
\item The Japanese bracket $\langle\cdot\rangle$ is defined for all $\xi\in\mathbb R^n$ by $\langle\xi\rangle = \sqrt{1+\vert\xi\vert^2}$.
\item For all continuous function $F:\mathbb R^n\rightarrow\mathbb C$, the notation $F(D_x)$ is used to denote the Fourier multiplier associated with the symbol $F(\xi)$.
\end{enumerate}

\section{Main results}\label{results}

This section is devoted to present the main results contained in this work.

\subsection{Regularizing effects} We begin by sharply describing the smoothing properties of the evolution operators $e^{-t\mathcal P}$ generated by the fractional Ornstein-Uhlenbeck operator $\mathcal P$. To that end, we consider the finite-dimensional vector space $S\subset\mathbb R^n$ defined by the following intersection of kernels
\begin{equation}\label{032019E2}
	S = \bigcap_{j=0}^{n-1}\Ker\big(\sqrt Q(B^T)^j\big).
\end{equation}
The structure of this vector space $S$, intrinsically linked to the matrices $B$ and $Q$, will play a key role in the study of fractional Ornstein-Uhlenbeck operators and their associated semigroups. Among others, we will prove that its canonical Euclidean orthogonal contains the smoothing directions of the operators $e^{-t\mathcal P}$, i.e. $\Sigma = S^{\perp}$ is the unique maximal vector space answering the problem presented in the introduction. By definition, we may consider the smallest integer $0\le r\le n-1$ satisfying
\begin{equation}\label{22072019E4}
	S = \bigcap_{j=0}^r\Ker\big(\sqrt Q(B^T)^j\big).
\end{equation}
This integer $0\le r \le n-1$ will also play a key role in the following.

First, we begin by checking that any smoothing direction of the operators $e^{-t\mathcal P}$ is contained in $S^{\perp}$ the canonical Euclidean orthogonal of the vector space $S$.

\begin{thm}\label{17072019T1} Let $\mathcal P$ be the fractional Ornstein-Uhlenbeck operator defined in \eqref{22072019E2} and equipped with the domain \eqref{17062020E2}. We consider $S$ the vector space associated with the matrices $B$ and $Q$ defined in \eqref{032019E2}. If there exist $t>0$ and $\xi_0\in\mathbb R^n$ such that the operator $\langle\xi_0,\nabla_x\rangle e^{-t\mathcal P}$ is bounded on $L^2(\mathbb R^n)$, then $\xi_0\in S^{\perp}$.
\end{thm}

Notice that the directions of $\mathbb R^n$ in which the evolution operators $e^{-t\mathcal P}$ can be differentiated depend on the matrices $B$ and $Q$ but not on the positive real number $s>0$.

In order to sharply describe the constant $c_{t,m,\xi_1,\ldots,\xi_m}>0$ appearing in \eqref{17062020E1} with respect to the vectors $\xi_1,\ldots,\xi_m$ in short times $0<t\ll1$, we need to introduce the notion of index of any vector of $S^{\perp}$. To that end, we consider the finite-dimensional vector spaces $V_0,\ldots,V_r\subset\mathbb R^n$ defined for all $k\in\{0,\ldots,r\}$ by
\begin{equation}\label{18072019E3}
	V_k = \bigcap_{j=0}^k\Ker\big(\sqrt Q(B^T)^j\big).
\end{equation}
Notice that the orthogonal complement of these vector spaces $V^{\perp}_0,\ldots, V^{\perp}_r$ form an increasing family which satisfies, according to \eqref{22072019E4},
\begin{equation}\label{18072019E4}
	V_0^{\perp}\subsetneq\ldots\subsetneq V_r^{\perp} = S^{\perp}.
\end{equation}
This stratification of $S^{\perp}$ allows to define the \textit{index} of any vector $\xi_0\in S^{\perp}$ by
\begin{equation}\label{18072019E2}
	k_{\xi_0} = \min\big\{0\le k\le r : \xi_0\in V^{\perp}_k\big\}.
\end{equation}
This notion of index plays a key role in the understanding of the blow-up for short times of the seminorms associated with the smoothing effects of the evolution operators $e^{-t\mathcal P}$, as illustrated in the following theorem which is the main result contained in this work.

\begin{thm}\label{09072019T1} Let $\mathcal P$ be the fractional Ornstein-Uhlenbeck operator defined in \eqref{22072019E2} and equipped with the domain \eqref{17062020E2}. We consider $S$ the vector space associated with the matrices $B$ and $Q$, and $0\le r\le n-1$ the smallest integer such that \eqref{22072019E4} holds. There exist some positive constants $c>1$ and $T>0$ such that for all $m\geq1$, $\xi_1,\ldots,\xi_m\in S^{\perp}$, $0<t<T$ and $u\in L^2(\mathbb R^n)$,
$$\big\Vert\langle\xi_1,\nabla_x\rangle\ldots\langle\xi_m,\nabla_x\rangle e^{-t\mathcal P}u\big\Vert_{L^2}\le\frac {c^me^{\frac12\Tr(B)t}}{t^{k_{\xi_1}+\ldots+k_{\xi_m}+\frac m{2s}}}\ \bigg(\prod_{j=1}^m\vert\xi_j\vert\bigg)\ (m!)^{\frac 1{2s}}\ \Vert u\Vert_{L^2},$$
where $0\le k_{\xi_j}\le r$ denotes the index of the vector $\xi_j\in S^{\perp}$.
\end{thm}

This result shows that the semigroup $(e^{-t\mathcal P})_{t\geq0}$ enjoys partial Gevrey-type smoothing properties, the short-time asymptotics of $m$ differentiations of this semigroup in the directions generated by the vectors $\xi_1,\ldots,\xi_m\in S^{\perp}$ being given by $\mathcal O(t^{-k_{\xi_1}-\ldots-k_{\xi_m}-\frac m{2s}})$, and therefore depend on the indexes $k_{\xi_1},\ldots,k_{\xi_m}$ of those vectors. The justification of the terminology {\og}partial Gevrey-type smoothing properties{\fg} is the following. In the particular case where the vector space $S$ is reduced to $\{0\}$, i.e. when the Kalman rank condition \eqref{22072019E1} holds according to Lemma \ref{18072019P3}, Theorem \ref{09072019T1} implies that the evolution operators $e^{-t\mathcal P}$ can be differentiated in any direction of $\mathbb R^n$. The behavior of the estimates given by Theorem \ref{09072019T1} with respect to $m\geq1$ shows that the semigroup $(e^{-t\mathcal P})_{t\geq0}$ enjoys Gevrey-type regularity for all $t>0$, i.e. Gevrey regularity when $0<s<1/2$, analytic regularity when $s=1/2$ and ultra-analytic regularity when $s>1/2$. In the general case, the evolution operators $e^{-t\mathcal P}$ can be differentiated only in specific directions of $\mathbb R^n$ and one can say that the semigroup $(e^{-t\mathcal P})_{t\geq0}$ only enjoys {\og}partial Gevrey-type regularity{\fg} for all $t>0$.

Theorem \ref{09072019T1} is the exact extension of \cite{AB2} (Example 2.11) in the particular case of semigroups generated by Ornstein-Uhlenbeck operators (case $s=1$). More generally, one of the objectives of the work \cite{AB2} is to describe the smoothing properties of the evolution operators $e^{-tq^w}$ generated by accretive quadratic operators $q^w(x,D_x)$. Theorem 2.6 in \cite{AB2}, which is quite a lot similar to Theorem \ref{09072019T1}, states that there exist some positive constants $c>1$ and $T>0$ such that for all $0<t\le T$, $m\geq1$, $X_1,\ldots,X_m\in S^{\perp}$ and $u\in L^2(\mathbb R^n)$,
$$\big\Vert\langle X_1,X\rangle^w\ldots\langle X_m,X\rangle^we^{-tq^w}u\big\Vert_{L^2}\le\frac {c^m}{t^{k_{\xi_1}+\ldots+k_{\xi_m}+\frac m2}}\ \bigg(\prod_{j=1}^m\vert\xi_j\vert\bigg)\ \sqrt{m!}\ \Vert u\Vert_{L^2},$$
with
$$\langle X_j,X\rangle^w = \langle x_j,x\rangle + \langle\xi_j,D_x\rangle,\quad  X_j = (x_j,\xi_j)\in S^{\perp},\quad D_x = -i\nabla_x,$$
where $S\subset \mathbb R^{2n}$ is a vector subspace of the phase space intrinsically linked the accretive quadratic operator $q^w(x,D_x)$, and $k_{X_1},\ldots,k_{X_m}$ are the indexes of the vectors $X_1,\ldots,X_m\in S^{\perp}$, similarly defined than \eqref{18072019E2}. These estimates for accretive quadratic semigroups are obtained by exploiting the polar decomposition of the evolution operators $e^{-tq^w}$, i.e. a splitting formula for these operators as the product of a selfadjoint operator and a unitary operator in $L^2(\mathbb R^n)$, which the main result of the paper \cite{AB2} (Theorem 2.1). The strength of Theorem \ref{09072019T1} is to obtain the same kind of estimates for all $s>0$ and to track the various constants with respect to this parameter. This theorem will be proven by using the following explicit formula for the operators $e^{-t\mathcal P}$, coming from the work \cite{AB} (Theorem 1.1), which can also be seen as a polar decomposition:
$$e^{-t\mathcal P} = \exp\bigg(-\frac t2\int_0^1\big\vert\sqrt Qe^{\alpha tB^T}D_x\big\vert^{2s}\ \mathrm d\alpha\bigg)\ e^{-t\langle Bx,\nabla_x\rangle}.$$
The strategy of using representations of the above type has already been used in the literature. For example, in the work \cite{MR3906169}, it allowed to derive some $L^p$ estimates for the maximal fractional regularity for degenerate non-local Kolmogorov operators.

Assuming that the vector space $S$ is reduced to zero anew, we notice that Theorem \ref{09072019T1} and a straightforward induction imply in particular that for all $\alpha\in\mathbb N^n$, $0<t<T$ and $u\in L^2(\mathbb R^n)$,
\begin{equation}\label{09092019E3}
	\big\Vert\partial^{\alpha}_x(e^{-t\mathcal P}u)\big\Vert_{L^2(\mathbb R^n)}\le\frac{c^{1+\vert\alpha\vert}}{t^{\vert\alpha\vert(\frac1{2s}+r)}}\ (\alpha!)^{\frac1{2s}}\ \Vert u\Vert_{L^2(\mathbb R^n)},
\end{equation}
since the index of any vector $\xi_0\in S^{\perp}$ satisfies $0\le k_{\xi_0}\le r$. These Gevrey-type regularizing effect estimates have already been proven in \cite{AB} (Theorem 1.1), with the same blow-up of the associated seminorms for short times $t\rightarrow0^+$.

In a general setting on the matrices $B$ and $Q$, Theorem \ref{09072019T1} does not provide any long time behavior of the norm of the linear operators $\langle\xi_1,\nabla_x\rangle\ldots\langle\xi_m,\nabla_x\rangle e^{-t\mathcal P}$. However, this long-time study can be performed in particular cases. For example, notice from Plancherel's theorem and a straightforward study of function that for all $m\geq1$, $\xi_1,\ldots,\xi_m\in\mathbb R^n$, $t>0$ and $u\in L^2(\mathbb R^n)$,
$$\big\Vert\langle\xi_1,\nabla_x\rangle\ldots\langle\xi_m,\nabla_x\rangle e^{-t(-\Delta_x)^s}u\big\Vert_{L^2}\le\frac {2^{-\frac m{2s}}}{t^{\frac m{2s}}}\ \bigg(\prod_{j=1}^m\vert\xi_j\vert\bigg)\ (m!)^{\frac 1{2s}}\ \Vert u\Vert_{L^2}.$$
We can therefore completely describe the long-time asymptotics of the smoothing properties of semigroups generated by fractional Laplacians. By adapting the proof of Theorem \ref{09072019T1}, one can obtain estimates in the more general case where the matrix $B$ is nilpotent, which apply in particular to semigroups generated by fractional Kolmogorov operators, some operators naturally appearing in the theory of kinetic equations defined as follows. Given $s>0$, the associated fractional Kolmogorov operator $\mathcal K$ is the operator defined by
\begin{equation}\label{22062020E7}
	\mathcal K = (-\Delta_v)^s + v\cdot\nabla_x,\quad(x,v)\in\mathbb R^{2n}.
\end{equation}
This is the fractional Ornstein-Uhlenbeck operator associated with the matrices $B$ and $Q$ respectively given by
$$B = \begin{pmatrix}
	0_n & I_n \\
	0_n & 0_n
\end{pmatrix}\quad \text{and}\quad Q = 2^{\frac1s}\begin{pmatrix}
	0_n & 0_n \\
	0_n & I_n
\end{pmatrix}.$$
We study the operator $\mathcal K$ equipped with the domain
\begin{equation}\label{22062020E8}
	D(\mathcal K) = \big\{u\in L^2(\mathbb R^{2n}) : \mathcal Ku\in L^2(\mathbb R^{2n})\big\}.
\end{equation}
A straightforward calculus shows that the vector space $S$ associated with the matrices $B$ and $Q$ is given by
$$S = \Ker\big(\sqrt Q\big)\cap\Ker\big(\sqrt QB^T\big) = \{0\}.$$
The associated integer $0\le r\le 2n-1$ defined in \eqref{22072019E4} is therefore equal to $1$. The smoothing properties of the semigroup generated by the operator $\mathcal K$ are presented in the following result.

\begin{prop}\label{22062020P1} Let $\mathcal K$ be the fractional Kolmogorov operator defined in \eqref{22062020E7} and equipped with the domain \eqref{22062020E8}. There exists a positive constant $c>1$ such that for all $(\alpha,\beta)\in\mathbb N^{2n}$, $t>0$ and $u\in L^2(\mathbb R^{2n})$,
$$\big\Vert\partial^{\alpha}_x\partial^{\beta}_v(e^{-t\mathcal K}u)\big\Vert_{L^2}\le\frac{c^{\vert\alpha\vert+\vert\beta\vert}}{t^{(1+\frac1{2s})\vert\alpha\vert
+ \frac{\vert\beta\vert}{2s}}}\ (\alpha!)^{\frac1{2s}}\ (\beta!)^{\frac1{2s}}\ \Vert u\Vert_{L^2}.$$
\end{prop}

The long-time study of the smoothing properties of semigroups generated by fractional Ornstein-Uhlenbeck operators is an interesting question that remains open.

\subsection{Subelliptic estimates}\label{subelliptic} The study of the smoothing properties of semigroups generated by fractional Ornstein-Uhlenbeck operators allows to obtain partial subelliptic estimates enjoyed by these operators on the whole space.

\begin{thm}\label{03072019T2} Let $\mathcal P$ be the fractional Ornstein-Uhlenbeck operator defined in \eqref{22072019E2} and equipped with the domain \eqref{17062020E2}. We consider $0\le r\le n-1$ the smallest integer such that \eqref{22072019E4} holds. There exists a positive constant $c>0$ such that for all $u\in D(\mathcal P)$,
$$\sum_{k=0}^r\big\Vert\langle\sqrt Q(B^T)^kD_x\rangle^{\frac{2s}{1+2ks}}u\big\Vert_{L^2}\le c\big(\Vert\mathcal Pu\Vert_{L^2} + \Vert u\Vert_{L^2}\big).$$
\end{thm}

The entire proof of Theorem \ref{03072019T2} will not be presented in the present paper since it is completely detailed in the author's thesis work \cite{A} (Chapter 4, Subsection 4.3). Nevertheless, let us present its main steps. Theorem \ref{03072019T2} can be directly obtained by exploiting the following estimates, established in \eqref{18062020E5} in Subsection \ref{refeffects} while proving Theorem \ref{09072019T1}, holding for all $k\in\{0,\ldots,r\}$, $q>0$, $0<t\ll1$ and $u\in L^2(\mathbb R^n)$,
\begin{equation}\label{22062020E1}
	\big\Vert\big\vert\sqrt Q(B^T)^kD_x\big\vert^qe^{-t\mathcal P}u\big\Vert_{L^2}\le\bigg(\frac c{t^{k+\frac1{2s}}}\bigg)^q\ e^{\frac12\Tr(B)t}\ q^{\frac q{2s}}\ \Vert u\Vert_{L^2},
\end{equation}
and elements of interpolation theory. Precisely, the link between the above smoothing estimates and the subelliptic estimates stated in Theorem \ref{03072019T2} is made by the following result

\begin{prop}[Proposition 2.7 in \cite{MR3710672}]\label{14062020P1} Let $X$ be a Hilbert space and $A:D(A)\subset X\rightarrow X$ be a maximal accretive operator such that $(-A,D(A))$ is the generator of a strongly continuous semigroup $(T(t))_{t\geq0}$. Assume that there exists a Banach space $E\subset X$, $\rho>1$ and $C>0$ such that
$$\forall t>0,\quad \Vert T(t)\Vert_{\mathcal L(X,E)}\le\frac C{t^{\rho}},$$
and that $t\mapsto T(t)u$ is measurable with values in $E$ for each $u\in X$. Then, the following continuous inclusion holds
$$D(A)\subset(X,E)_{\frac1\rho,2},$$
where $(X,E)_{\frac1\rho,2}$ denotes the space obtained by real interpolation.
\end{prop}

The estimate \eqref{22062020E1} used with $q = 1+\lfloor2s\rfloor$ and Proposition \ref{14062020P1} applied with the Hilbert spaces $X = L^2(\mathbb R^n)$ and
$$E = \mathscr H_k = \big\{u\in L^2(\mathbb R^n) : \Lambda_k^qu\in L^2(\mathbb R^n)\big\},\quad\text{with}\quad\Lambda_k = \langle\sqrt Q(B^T)^kD_x\rangle,$$
allow to obtain that for all $k\in\{0,\ldots,r\}$, the following continuous inclusion holds
\begin{equation}\label{22062020E2}
	D(\mathcal P)\subset(L^2(\mathbb R^n),\mathscr H_k)_{\theta,2}\quad\text{with}\quad\theta = \frac{2s}{q(1+2ks)}\in(0,1).
\end{equation}
By setting $\mathscr H_k$ as the domain of the Fourier multiplier $\Lambda_k^q$, another interpolation result, namely \cite{MR2523200} (Theorem 4.36), justifies that
\begin{equation}\label{22062020E3}
	(L^2(\mathbb R^n),\mathscr H_k)_{\theta,2} = (D((\Lambda_k^q)^0),D(\Lambda_k^q)^1))_{\theta,2} = D(\Lambda_k^{q\theta}).
\end{equation}
The combinaison of the continuous inclusion \eqref{22062020E2} and \eqref{22062020E3} ends the proof of Theorem \ref{03072019T2}. Notice that this proof follows line to line the one presented in Subsections 5.1 and 5.2 in \cite{AB}, where subelliptic estimates similar to the ones states in Theorem \ref{03072019T2} are established when the Kalman rank condition \eqref{22072019E1} holds. This strategy of proof, inspired by the work \cite{MR3710672}, is also used in the paper \cite{AB2} to establish subelliptic estimates enjoyed by accretive quadratic operators on the whole space, of which Ornstein-Uhlenbeck operators (case $s=1$) are a particular case, see \cite{AB2} (Example 2.11).

Let us assume that the Kalman rank condition \eqref{22072019E1} holds. In this case, the vector space $S$ is reduced to zero as we check in Lemma \ref{29082018E1} in Appendix, and it follows that
$$\forall\xi\in\mathbb R^n,\quad\vert\xi\vert^2\lesssim\sum_{k=0}^r\big\vert\sqrt Q(B^T)^k\xi\big\vert^2.$$
We therefore deduce from this estimate, Theorem \ref{03072019T2} and Plancherel's theorem that for all $u\in D(\mathcal P)$,
\begin{align*}
	\big\Vert\langle D_x\rangle^{\frac{2s}{1+2rs}}u\big\Vert_{L^2} 
	& \lesssim\sum_{k=0}^r\big\Vert\langle\sqrt Q(B^T)^k D_x\rangle^{\frac{2s}{1+2rs}}u\big\Vert_{L^2} \\
	& \lesssim\sum_{k=0}^r\big\Vert\langle\sqrt Q(B^T)^k D_x\rangle^{\frac{2s}{1+2ks}}u\big\Vert_{L^2} \lesssim \Vert\mathcal Pu\Vert_{L^2} + \Vert u\Vert_{L^2}.
\end{align*}
When $s\geq1$ is a positive integer, the operator $\mathcal P$ therefore enjoys a global subelliptic estimate on the whole space
$$\big\Vert\langle D_x\rangle^{2s(1-\delta)}u\big\Vert_{L^2}\lesssim\Vert \mathcal Pu\Vert_{L^2} + \Vert u\Vert_{L^2},$$
with a loss of 
$$\delta = \frac{2rs}{1+2rs}>0,$$ 
derivatives with respect to the elliptic case. The above estimate has already been established in \cite{AB} (Corollary 1.15).

\section{Regularizing effects of semigroups generated by fractional Ornstein-Uhlenbeck operators}\label{reg}

Let $\mathcal P$ be the fractional Ornstein-Uhlenbeck operator defined in \eqref{22072019E2} and equipped with the domain \eqref{17062020E2}. We consider the vector space $S\subset\mathbb R^n$ defined in \eqref{032019E2} and $0\le r\le n-1$ the smallest integer such that \eqref{22072019E4} holds. This section is devoted to the study of the smoothing properties of the semigroup $(e^{-t\mathcal P})_{t\geq0}$ generated by the operator $\mathcal P$. To that end, we will use the fact that the evolution operators $e^{-t\mathcal P}$ are explicitly given by the following explicit formula
\begin{equation}\label{18072019E1}
	e^{-t\mathcal P} = \exp\bigg(-\frac t2\int_0^1\big\vert\sqrt Qe^{\alpha tB^T}D_x\big\vert^{2s}\ \mathrm d\alpha\bigg)\ e^{-t\langle Bx,\nabla_x\rangle}.
\end{equation}
This formula is a straightforward reformulation of \cite{AB} (Theorem 1.1). Indeed, this result states that the evolution operators $e^{-t\mathcal P}$ are given through their Fourier transform for all $t\geq0$ by
$$\forall u\in L^2(\mathbb R^n),\quad \widehat{e^{-t\mathcal P}u} = \exp\bigg(-\frac12\int_0^t\big\vert\sqrt Qe^{\tau B^T}\cdot\big\vert^{2s}\ \mathrm d\tau\bigg)\ e^{\Tr(B)t}\ \widehat u(e^{tB^T}\cdot).$$
Moreover, $e^{-t\langle Bx,\nabla_x\rangle}u = u(e^{-tB}\cdot)$ for all $u\in L^2(\mathbb R^n)$ and a straightforward calculus of inverse Fourier transform shows that for all $\xi\in\mathbb R^n$,
$$\widehat{u(e^{-tB}\cdot)}(\xi) = \int_{\mathbb R^n}e^{-i\langle x,\xi\rangle}u(e^{-tB}x)\ \mathrm dx = e^{\Tr(B)t}\ \widehat u(e^{tB^T}\xi).$$
This proves that formula \eqref{18072019E1} actually holds. It shows that the operators $e^{-t\mathcal P}$ can be split as the product of a selfadjoint Fourier multiplier and an operator which is a similitude on $L^2(\mathbb R^n)$, since
\begin{equation}\label{18062020E3}
	\forall u\in L^2(\mathbb R^n),\quad\big\Vert e^{-t\langle Bx,\nabla_x\rangle}u\big\Vert_{L^2} = \big\Vert u(e^{-tB}\cdot)\big\Vert_{L^2(\mathbb R^n)} = e^{\frac12\Tr(B)t}\ \Vert u\Vert_{L^2}.
\end{equation}
The splitting formula \eqref{18072019E1} can somehow be seen as the polar decomposition of the evolution operators $e^{-t\mathcal P}$.

\subsection{Smoothing directions} The first part of this section is devoted to the proof of Theorem \ref{17072019T1}. Let $t>0$ and $\xi_0\in\mathbb R^n$. We assume that the linear operator $\langle\xi_0,\nabla_x\rangle e^{-t\mathcal P}$ is bounded on $L^2(\mathbb R^n)$. We aim at proving that $\xi_0\in S^{\perp}$. Let us introduce the time-dependent real-valued symbol $a_t$ defined for all $\xi\in\mathbb R^n$ by
\begin{equation}\label{18072019E14}
	a_t(\xi) = \frac12\int_0^1\big\vert\sqrt Qe^{\alpha tB^T}\xi\big\vert^{2s}\ \mathrm d\alpha.
\end{equation}
With this notation, the polar decomposition \eqref{18072019E1} can be written in the following way
\begin{equation}\label{18072019E18}
	e^{-t\mathcal P} = e^{-ta^w_t}e^{-t\langle Bx,\nabla_x\rangle},
\end{equation}
where $e^{-ta^w_t}$ denotes the time-dependent Fourier multiplier associated with the time-dependent real-valued symbol $e^{-ta_t}$. We first notice from the splitting formula \eqref{18072019E18} that the boundedness on $L^2(\mathbb R^n)$ of the operator $\langle\xi_0,\nabla_x\rangle e^{-t\mathcal P}$ is equivalent to the boundedness of the Fourier multiplier $\langle\xi_0,\nabla_x\rangle e^{-ta^w_t}$, since the evolution operator $e^{-t\langle Bx,\nabla_x\rangle}$ is invertible on $L^2(\mathbb R^n)$. As a consequence, there exists a positive constant $c_{t,\xi_0}>0$ depending on $t>0$ and $\xi_0\in\mathbb R^n$ such that
\begin{equation}\label{18072019E11}
	\forall u\in L^2(\mathbb R^n),\quad \big\Vert\langle\xi_0,\nabla_x\rangle e^{-ta^w_t}u\big\Vert_{L^2}\le c_{t,\xi_0}\Vert u\Vert_{L^2}.
\end{equation}
According to the orthogonal decomposition $\mathbb R^n = S\oplus S^{\perp}$, we write $\xi_0 = \xi_{0,S} + \xi_{0,S^{\perp}}$, with $\xi_{0,S}\in S$ and $\xi_{0,S^{\perp}}\in S^{\perp}$. For all $\lambda\geq0$, we consider the Gaussian function $u_{\lambda}\in\mathscr S(\mathbb R^n)$ defined for all $x\in\mathbb R^n$ by
\begin{equation}\label{18072019E12}
	u_{\lambda}(x) = e^{i\lambda\langle\xi_{0,S},x\rangle}e^{-\vert x\vert^2}.
\end{equation}
The strategy will be to obtain upper and lower bounds for the norm
$$\big\Vert\langle\xi_0,\nabla_x\rangle e^{-ta^w_t}u_{\lambda}\big\Vert_{L^2},$$
and to consider the asymptotics when $\lambda$ goes to $+\infty$ in order to conclude that $\xi_{0,S}\in S$ has to be equal to zero. An upper bound can be directly obtained since it follows from \eqref{18072019E11}, \eqref{18072019E12} and the Cauchy-Schwarz inequality that for all $\lambda\geq0$,
\begin{equation}\label{18072019E13}
	\big\Vert\langle\xi_0,\nabla_x\rangle e^{-ta^w_t}u_{\lambda}\big\Vert_{L^2}\le c_{t,\xi_0}\Vert u_{\lambda}\Vert_{L^2} = c_{t,\xi_0}\Vert u_0\Vert_{L^2}.
\end{equation}
Notice that the right-hand side of the above estimate does not depend on the parameter $\lambda\geq0$. On the other hand, by using that for all $\lambda\geq0$, the Fourier transform of the function $u_{\lambda}$ is given for all $\xi\in\mathbb R^n$ by
\begin{equation}\label{02092019E4}
	\widehat{u_{\lambda}}(\xi) = \int_{\mathbb R^n}e^{-i\langle x,\xi\rangle}e^{i\lambda\langle\xi_{0,S},x\rangle}e^{-\vert x\vert^2}\ \mathrm dx
	= \int_{\mathbb R^n}e^{-i\langle x,\xi-\lambda\xi_{0,S}\rangle}e^{-\vert x\vert^2}\ \mathrm dx = \widehat{u_0}(\xi-\lambda\xi_{0,S}),
\end{equation}
we deduce from Plancherel's theorem that for all $\lambda\geq0$,
\begin{equation}\label{18072019E15}
	\big\Vert\langle\xi_0,\nabla_x\rangle e^{-ta^w_t}u_{\lambda}\big\Vert_{L^2}
	= \frac1{(2\pi)^{\frac n2}}\big\Vert\langle\xi_0,\xi\rangle e^{-ta_t(\xi)}\widehat{u_0}(\xi-\lambda\xi_{0,S})\big\Vert_{L^2}.
\end{equation}
Moreover, notice that for all $\xi\in\mathbb R^n$ and $\eta\in S$, 
\begin{equation}\label{18062020E1}
	a_t(\xi+\eta) = \frac12\int_0^1\big\vert\sqrt Qe^{\alpha tB^T}(\xi+\eta)\big\vert^{2s}\ \mathrm d\alpha = \frac12\int_0^1\big\vert\sqrt Qe^{\alpha tB^T}\xi\big\vert^{2s}\ \mathrm d\alpha = a_t(\xi).
\end{equation}
Indeed, the Cayley-Hamilton theorem applied to the matrix $B^T$ shows that
$$\forall j\in\mathbb N, \forall\xi\in\mathbb R^n,\quad (B^T)^j\xi\in\Span(\xi,B^T\xi,\ldots,(B^T)^{n-1}\xi).$$
This proves that the space $S$ is also given by the following infinite intersection of kernels
\begin{equation}\label{09092019E4}
	S = \bigcap_{j=0}^{+\infty}\Ker\big(\sqrt Q(B^T)^j\big).
\end{equation}
As a consequence of this fact, we get that
\begin{equation}\label{18062020E4}
	\forall\alpha\in[0,1],\forall\eta\in S,\quad \sqrt Qe^{\alpha tB^T}\eta = \sum_{j=0}^{+\infty}\frac{(\alpha t)^j}{j!}\sqrt Q(B^T)^j\eta = 0.
\end{equation}
This implies that \eqref{18062020E1} actually holds. We therefore deduce from \eqref{02092019E4}, \eqref{18072019E15}, \eqref{18062020E1}, a change of variable in the norm and the triangle inequality that for all $\lambda\geq0$,
\begin{align}\label{18072019E16}
	\big\Vert\langle\xi_0,\nabla_x\rangle e^{-ta^w_t}u_{\lambda}\big\Vert_{L^2}
	& = \frac1{(2\pi)^{\frac n2}}\big\Vert\langle\xi_0,\xi+\lambda\xi_{0,S}\rangle e^{-ta_t(\xi+\lambda\xi_{0,S})}\widehat{u_0}(\xi)\big\Vert_{L^2}, \\[5pt]
	& = \frac1{(2\pi)^{\frac n2}}\big\Vert\langle\xi_0,\xi+\lambda\xi_{0,S}\rangle e^{-ta_t(\xi)}\widehat{u_0}\big\Vert_{L^2} \nonumber \\[5pt]
	& \geq \big\Vert\langle\xi_0,\lambda\xi_{0,S}\rangle e^{-ta^w_t}u_0\big\Vert_{L^2} - \big\Vert\langle\xi_0,\nabla_x\rangle e^{-ta^w_t}u_0\big\Vert_{L^2}Ê\nonumber \\[5pt]
	& \geq \lambda\vert\xi_{0,S}\vert^2\big\Vert e^{-ta^w_t}u_0\big\Vert_{L^2} - c_{t,\xi_0}\Vert u_0\Vert_{L^2}, \nonumber
\end{align}
since $\lambda\xi_{0,S}\in S$. It follows from \eqref{18072019E13} and \eqref{18072019E16} that for all $\lambda\geq0$,
$$\lambda\vert\xi_{0,S}\vert^2\big\Vert e^{-ta^w_t}u_0\big\Vert_{L^2}\le 2c_{t,\xi_0}\Vert u_0\Vert_{L^2}.$$
Since the function $e^{-ta^w_t}u_0$ is not equal to zero ($u_0$ is a Gaussian function and the symbol of the bounded Fourier multiplier $e^{-ta^w_t}$ is not equal to zero) and that the right-hand side of the above estimate does not depend on the parameter $\lambda\geq0$, we conclude that $\xi_{0,S} = 0$, that is, $\xi_0\in S^{\perp}$. This ends the proof of Theorem \ref{17072019T1}.

\subsection{Regularizing effects}\label{refeffects} In the second part of this section, we prove Theorem \ref{09072019T1}. We keep the notations introduced in the previous subsection, in particular the time-dependent real-valued symbol $a_t$ defined in \eqref{18072019E14}. 

Since the operator $e^{-t\langle Bx,\nabla_x\rangle}$ is a similitude on $L^2(\mathbb R^n)$ for all $t\geq0$, we notice from \eqref{18072019E18} that it is sufficient to obtain the regularizing effects of the operators $e^{-ta^w_t}$ to derive the ones of the operators $e^{-t\mathcal P}$. Let $m\geq1$ be a positive integer and $\xi_1,\ldots,\xi_m\in S^{\perp}$. We are interested in studying the operators 
$$\langle\xi_1,\nabla_x\rangle\ldots\langle\xi_m,\nabla_x\rangle e^{-ta^w_t},\quad t>0.$$ 
Since the operators $\langle\xi_j,\nabla_x\rangle$ and $e^{-ta^w_t}$ are Fourier multipliers, they commute, and the following factorization holds for all $t>0$,
\begin{align}\label{18072019E6}
	\langle\xi_1,\nabla_x\rangle\ldots\langle\xi_m,\nabla_x\rangle e^{-ta^w_t} 
	& = \langle\xi_1,\nabla_x\rangle\ldots\langle\xi_m,\nabla_x\rangle\underbrace{e^{-\frac tma^w_t}\ldots e^{-\frac tma^w_t}}_{\text{$m$ factors}} \\
	& = \langle\xi_1,\nabla_x\rangle e^{-\frac tma^w_t}\ldots\langle\xi_m,\nabla_x\rangle e^{-\frac tma^w_t}, \nonumber
\end{align}
where we used the semigroup property of the family of contraction operators $(e^{-s a^w_t})_{s\geq0}$. The initial problem is therefore reduced to the analysis of the operators 
$$\langle\xi_j,\nabla_x\rangle e^{-\frac tma^w_t},\quad j\in\{1,\ldots,m\},\ t>0.$$ 
In the following proposition, we prove a first regularizing effect for the semigroups $(e^{-\tau a^w_t})_{\tau\geq0}$ for all $t>0$.

\begin{prop}\label{18072019P2} There exist some positive constants $c>1$ and $T>0$ such that for all $k\in\{0,\ldots,r\}$, $q>0$, $0<t<T$, $\tau>0$ and $u\in L^2(\mathbb R^n)$,
$$\big\Vert\big\vert\sqrt Q(B^T)^kD_x\big\vert^qe^{-\tau a^w_t}u\big\Vert_{L^2}\le\bigg(\frac c{t^k\tau^{\frac1{2s}}}\bigg)^q\ q^{\frac q{2s}}\ \Vert u\Vert_{L^2}.$$
\end{prop}

\begin{proof} Let us begin by noticing that to prove Proposition \ref{18072019P2}, it is sufficient to obtain the existence of some positive constants $c_1>0$ and $t_1>0$ such that for all $q>0$, $0<t<t_1$, $\tau>0$ and $\xi\in\mathbb R^n$,
\begin{equation}\label{09072019E1}
\bigg(\int_0^1\big\vert\sqrt Qe^{\alpha tB^T}\xi\big\vert^2\ \mathrm d\alpha\bigg)^{\frac q2}\exp\bigg(-\frac{\tau}2\int_0^1\big\vert\sqrt Qe^{\alpha tB^T}\xi\big\vert^{2s}\ \mathrm d\alpha\bigg)
\le\bigg(\frac{c_1}{\tau^{\frac1{2s}}}\bigg)^q\bigg(\frac q{es}\bigg)^{\frac q{2s}}.
\end{equation}
Indeed, we check in Proposition \ref{18072019P3} in Appendix, by using finite-dimensional and compactness arguments, that there exist some positive constants $c_0>0$ and $t_0>0$ such that for all $0<t<t_0$ and $\xi\in\mathbb R^n$,
\begin{equation}\label{09072019E2}
	\int_0^1\big\vert\sqrt Qe^{\alpha tB^T}\xi\big\vert^2\ \mathrm d\alpha\geq c_0\sum_{k=0}^rt^{2k}\big\vert\sqrt Q(B^T)^k\xi\big\vert^2.
\end{equation}
Once the estimates \eqref{09072019E1} and \eqref{09072019E2} are established, we deduce from the definition \eqref{18072019E14} of the symbols $a_t$ and Plancherel's theorem that for all $k\in\{0,\ldots,r\}$, $q>0$, $0<t<\min(t_0,t_1)$, $\tau>0$ and $u\in L^2(\mathbb R^n)$,
\begin{align*}
	\big\Vert\big\vert\sqrt Q(B^T)^kD_x\big\vert^qe^{-\tau a^w_t}u\big\Vert_{L^2(\mathbb R^n)}
	& \le\bigg(\frac1{\sqrt{c_0}t^k}\bigg)^q\bigg\Vert\bigg(\int_0^1\big\vert\sqrt Qe^{\alpha tB^T}D_x\big\vert^2\ \mathrm d\alpha\bigg)^{\frac q2}e^{-\tau a^w_t}u\bigg\Vert_{L^2(\mathbb R^n)} \\[5pt]
	& \le\bigg(\frac1{\sqrt{c_0}t^k}\bigg)^q\ \bigg(\frac{c_1}{{\tau}^{\frac1{2s}}}\bigg)^q\ \bigg(\frac q{es}\bigg)^{\frac q{2s}}\ \Vert u\Vert_{L^2(\mathbb R^n)},
\end{align*}
which is the desired estimate. We therefore focus on establishing the estimate \eqref{09072019E1}. In order to simplify the notations, let us consider the real-valued symbol $\Gamma_{q,t,\tau}$ defined for all $q>0$, $t>0$, $\tau>0$ and $\xi\in\mathbb R^n$ by
\begin{equation}\label{09072019E3}
	\Gamma_{q,t,\tau}(\xi) = \bigg(\int_0^1\big\vert\sqrt Qe^{\alpha tB^T}\xi\big\vert^2\ \mathrm d\alpha\bigg)^{\frac q2}\exp\bigg(-\frac{\tau}2\int_0^1\big\vert\sqrt Qe^{\alpha tB^T}\xi\big\vert^{2s}\ \mathrm d\alpha\bigg).
\end{equation}
First of all, observe from \eqref{18062020E4} and \eqref{09072019E3} that for all $q>0$, $t>0$ and $\tau>0$, the symbol $\Gamma_{q,t,\tau}$ satisfies 
$$\forall\xi\in\mathbb R^n,\quad \Gamma_{q,t,\tau}(\xi) = \Gamma_{q,t,\tau}(\xi_{S^{\perp}}),$$
where $\xi_{S^{\perp}}\in\mathbb R^n$ denotes the coordinate of the vector $\xi\in\mathbb R^n$ with respect to the decomposition $\mathbb R^n = S\oplus S^{\perp}$, the orthogonality being taken with respect to the canonical Euclidean structure of $\mathbb R^n$. Thus, we only need to establish the estimate \eqref{09072019E1} when $\xi\in S^{\perp}$. To that end, we will take advantage of the homogeneity property with respect to the $\xi$-variable of the two integrals 
$$\int_0^1\big\vert\sqrt Qe^{\alpha tB^T}\xi\big\vert^2\ \mathrm d\alpha\quad\text{and}\quad\int_0^1\big\vert\sqrt Qe^{\alpha tB^T}\xi\big\vert^{2s}\ \mathrm d\alpha.$$
Let $\xi\in S^{\perp}\setminus\{0\}$ and $(\rho,\sigma)$ be the polar coordinates of $\xi$, i.e. $\xi = \rho\sigma$ with $\rho>0$ and $\sigma\in\mathbb S^{n-1}$. It follows from \eqref{09072019E3} and the estimate
$$\forall q>0,\forall x\geq0,\quad x^qe^{-x^{2s}} \le \bigg(\frac q{2es}\bigg)^{\frac q{2s}},$$
that for all $q>0$, $t>0$ and $\tau>0$,
\begin{align}\label{10052018E8}
	\Gamma_{q,t,\tau}(\xi) & = \bigg(\int_0^1\big\vert\sqrt Qe^{\alpha tB^T}\sigma\big\vert^2\ \mathrm d\alpha\bigg)^{\frac q2}\rho^q\exp\bigg(-\bigg(\frac{\tau}2\int_0^1\big\vert\sqrt Qe^{\alpha tB^T}\sigma\big\vert^{2s}\ \mathrm d\alpha\bigg)\rho^{2s}\bigg) \\[5pt]
	& \le \bigg(\int_0^1\big\vert\sqrt Qe^{\alpha tB^T}\sigma\big\vert^2\ \mathrm d\alpha\bigg)^{\frac q2}\bigg(\frac{\tau}2\int_0^1\big\vert\sqrt Qe^{\alpha tB^T}\sigma\big\vert^{2s}\ \mathrm d\alpha\bigg)^{-\frac q{2s}}\bigg(\frac q{2es}\bigg)^{\frac q{2s}} \nonumber \\[5pt]
	& \le \bigg(\frac{M_t}{\tau^{\frac 1{2s}}}\bigg)^q\bigg(\frac q{es}\bigg)^{\frac q{2s}}, \nonumber
\end{align}
where we set
\begin{equation}\label{17072019E1}
	M_t = \sup_{\eta\in\mathbb S^{n-1}\cap S^{\perp}}\bigg(\int_0^1\big\vert\sqrt Qe^{\alpha tB^T}\eta\big\vert^2\ \mathrm d\alpha\bigg)^{\frac12}\bigg(\int_0^1\big\vert\sqrt Qe^{\alpha tB^T}\eta\big\vert^{2s}\ \mathrm d\alpha\bigg)^{-\frac1{2s}}.
\end{equation}
The study of the term $M_t$ is made by the author and J. Bernier in \cite{AB} (Proposition 3.6) when the vector space $S$ is reduced to zero, that is, when the Kalman rank condition \eqref{22072019E1} holds according to Lemma \ref{29082018E1}. The same strategy of proof as the one used in this paper (finite-dimensional and compactness arguments) turns out to work in the general case. By using these arguments, it is proved in Proposition \ref{18072019P6} in Appendix that there exist some positive constants $c_1>0$ and $t_1>0$ such that for all $0<t<t_1$,
\begin{equation}\label{01052019E2}
	M_t\le c_1.
\end{equation}
We therefore deduce from \eqref{09072019E3}, \eqref{10052018E8} and \eqref{01052019E2} that for all $q>0$, $0<t<t_1$, $\tau>0$ and $\xi\in S^{\perp}$,
$$\bigg(\int_0^1\big\vert\sqrt Qe^{\alpha tB^T}\xi\big\vert^2\ \mathrm d\alpha\bigg)^{\frac q2}\exp\bigg(-\frac{\tau}2\int_0^1\big\vert\sqrt Qe^{\alpha tB^T}\xi\big\vert^{2s}\ \mathrm d\alpha\bigg)\le\bigg(\frac{c_1}{\tau^{\frac1{2s}}}\bigg)^q\bigg(\frac q{es}\bigg)^{\frac q{2s}}.$$
This ends the proof of the estimate \eqref{09072019E1} and hence, the one of Proposition \ref{18072019P2}.
\end{proof}

The above result also provides estimates for the operators $e^{-t\mathcal P}$. Indeed, we deduce from \eqref{18072019E1}, \eqref{18062020E3}, Proposition \ref{18072019P2} and Plancherel's theorem that for all $k\in\{0,\ldots,r\}$, $q>0$, $0<t<T$ and $u\in L^2(\mathbb R^n)$,
\begin{align}\label{18062020E5}
	\big\Vert\big\vert\sqrt Q(B^T)^kD_x\big\vert^qe^{-t\mathcal P}u\big\Vert_{L^2}
	& \le \bigg(\frac c{t^{k+\frac1{2s}}}\bigg)^q\ q^{\frac q{2s}}\ \big\Vert e^{-t\langle Bx,\nabla_x\rangle} u\big\Vert_{L^2} \\[5pt]
	& = \bigg(\frac c{t^{k+\frac1{2s}}}\bigg)^q\ e^{\frac12\Tr(B)t}\ q^{\frac q{2s}}\ \Vert u\Vert_{L^2}. \nonumber
\end{align}
As explained in Subsection \ref{subelliptic}, these estimates are key to establish the partial subelliptic estimates enjoyed by fractional Ornstein-Uhlenbeck operators on  the whole space.

\begin{cor}\label{18082019C1} There exist some positive constants $c>0$ and $0<T<1$ such that for all $\xi_0\in S^{\perp}$, $0<t<T$, $\tau>0$ and $u\in L^2(\mathbb R^n)$,
$$\big\Vert\langle\xi_0,\nabla_x\rangle e^{-\tau a^w_t}u\big\Vert_{L^2}\le c\vert\xi_0\vert\ t^{-k_{\xi_0}}{\tau}^{-\frac1{2s}}\ \Vert u\Vert_{L^2},$$
where $0\le k_{\xi_0}\le r$ denotes the index of the vector $\xi_0\in S^{\perp}$ defined in \eqref{18072019E2}.
\end{cor}

\begin{proof} First, let us check that there exists a positive constant $M_0>0$ such that for all $\xi_0\in S^{\perp}$ and $\xi\in\mathbb R^n$,
\begin{equation}\label{02092019E3}
	\big\vert\langle\xi_0,\xi\rangle\big\vert^2 \le M^2_0\vert\xi_0\vert^2\sum_{j=0}^{k_{\xi_0}}\big\vert\sqrt Q(B^T)^j\xi\big\vert^2.
\end{equation}
For all non-negative integer $k\in\{0,\ldots,r\}$, we consider $\mathbb P_k$ the orthogonal projection onto the vector subspace $V^{\perp}_k\subset\mathbb R^n$, the orthogonality being taken with respect to the canonical Euclidean structure of $\mathbb R^n$, where the vector subspace $V_k$ is defined in \eqref{18072019E3} by
$$V_k = \bigcap_{j=0}^k\Ker\big(\sqrt Q(B^T)^j\big).$$ 
Notice from the definition of the vector spaces $V_k$ that for all $k\in\{0,\ldots,r\}$, there exists a positive constant $c_k>0$ such that for all $\xi\in\mathbb R^n$,
$$\sum_{j=0}^k\big\vert\sqrt Q(B^T)^j\xi\big\vert^2 = \sum_{j=0}^k\big\vert\sqrt Q(B^T)^j\mathbb P_k\xi\big\vert^2\geq c_k\vert\mathbb P_k\xi\vert^2.$$
Let $\xi_0\in S^{\perp}$. The definition \eqref{18072019E2} of index implies that $\xi_0\in V^{\perp}_{k_{\xi_0}}$. It follows from the above estimate and the Cauchy-Schwarz inequality that for all $\xi\in\mathbb R^n$,
$$\big\vert\langle\xi_0,\xi\rangle\big\vert^2 = \big\vert\langle\xi_0,\mathbb P_{k_{\xi_0}}\xi\rangle\big\vert^2\le\vert\xi_0\vert^2\vert\mathbb P_{k_{\xi_0}}\xi\vert^2\le\frac{\vert\xi_0\vert^2}{c_{k_{\xi_0}}}\sum_{k=0}^{k_{\xi_0}}\big\vert\sqrt Q(B^T)^k\xi\big\vert^2.$$
This proves that the estimate \eqref{02092019E3} holds, with $M^2_0 = 1/\min_{k\in\{0,\ldots,r\}}c_k>0$. We therefore deduce from \eqref{02092019E3} and Plancherel's theorem that for all $\xi_0\in S^{\perp}$, $t\geq0$, $\tau\geq0$ and $u\in L^2(\mathbb R^n)$,
\begin{equation}\label{18072019E7}
	\big\Vert\langle\xi_0,\nabla_x\rangle e^{-\tau a^w_t}u\big\Vert_{L^2}\le M_0\vert\xi_0\vert\ \sum_{k=0}^{k_{\xi_0}}\big\Vert\big\vert\sqrt Q(B^T)^kD_x\big\vert e^{-\tau a^w_t}u\big\Vert_{L^2}.
\end{equation}
On the other hand, Proposition \ref{18072019P2} provides the existence of some positive constants $M_1>1$ and $t_1>0$ such that for all $k\in\{0,\ldots,r\}$, $0<t<t_1$, $\tau>0$ and $u\in L^2(\mathbb R^n)$,
\begin{equation}\label{18072019E8}
	\big\Vert\big\vert\sqrt Q(B^T)^kD_x\big\vert e^{-\tau a^w_t}u\big\Vert_{L^2}
	\le\frac{c'}{t^k\tau^{\frac1{2s}}}\ \Vert u\Vert_{L^2}.
\end{equation}
It follows from \eqref{18072019E7} and \eqref{18072019E8} that for all $\xi_0\in S^{\perp}$, $0<t<\min(t_1,1)$, $\tau>0$ and $u\in L^2(\mathbb R^n)$,
\begin{align*}
	\big\Vert\langle\xi_0,\nabla_x\rangle e^{-\tau a^w_t}u\big\Vert_{L^2}
	& \le M_0\vert\xi_0\vert\ \bigg(\sum_{k=0}^{k_{\xi_0}}\frac{M_1}{t^k\tau^{\frac1{2s}}}\bigg)\ \Vert u\Vert_{L^2} \\
	& \le M_0M_1(r+1)\vert\xi_0\vert\ t^{-k_{\xi_0}}\tau^{-\frac1{2s}}\ \Vert u\Vert_{L^2},
\end{align*}
since $0\le k_{\xi_0}\le r$. This ends the proof of Corollary \ref{18082019C1}.
\end{proof}

We can now tackle the proof of Theorem \ref{09072019T1}. To that end, we implement the strategy presented in the beginning of this subsection. Let $m\geq1$ and $\xi_1,\ldots,\xi_m\in S^{\perp}$. We denote by $0\le k_{\xi_j}\le r$ the index of the vector $\xi_j\in S^{\perp}$ defined in \eqref{18072019E2}. It follows from \eqref{18072019E6} that for all $t\geq0$, 
\begin{equation}\label{18072019E9}
	\langle\xi_1,\nabla_x\rangle\ldots\langle\xi_m,\nabla_x\rangle e^{-ta^w_t} 
	= \langle\xi_1,\nabla_x\rangle e^{-\frac tma^w_t}\ldots\langle\xi_m,\nabla_x\rangle e^{-\frac tma^w_t}.
\end{equation}
According to Corollary \ref{18082019C1}, there exist some positive constants $c>0$ and $0<T<1$ such that for all $\xi_0\in S^{\perp}$, $m\geq1$ and $0<t<T$,
\begin{equation}\label{18072019E10}
	\big\Vert\langle\xi_0,\nabla_x\rangle e^{-\frac tma^w_t}u\big\Vert_{L^2}\le\frac c{t^{k_{\xi_0}+\frac1{2s}}}\ \vert\xi_0\vert\ m^{\frac1{2s}}\ \Vert u\Vert_{L^2},
\end{equation}
where $0\le k_{\xi_0}\le r$ denotes the index of the vector $\xi_0\in S^{\perp}$ defined in \eqref{18072019E2}. We deduce from \eqref{18072019E9} and \eqref{18072019E10} that for all $0<t<T$ and $u\in L^2(\mathbb R^n)$,
\begin{align*}
	\big\Vert\langle\xi_1,\nabla_x\rangle\ldots\langle\xi_m,\nabla_x\rangle e^{-ta^w_t}u\big\Vert_{L^2}
	& \le \frac{c^m}{t^{k_{\xi_1}+\ldots+k_{\xi_m}+\frac m{2s}}}\ \bigg(\prod_{j=1}^m\vert\xi_j\vert\bigg)\ m^{\frac m{2s}}\ \Vert u\Vert_{L^2} \\[5pt]
	& \le \frac{e^{\frac m{2s}}c^m}{t^{k_{\xi_1}+\ldots+k_{\xi_m}+\frac m{2s}}}\ \bigg(\prod_{j=1}^m\vert\xi_j\vert\bigg)\ (m!)^{\frac1{2s}}\ \Vert u\Vert_{L^2},
\end{align*}
where we used the factorial estimate $m^m\le e^mm!$. It follows from \eqref{18072019E1}, \eqref{18062020E3} and Plancherel's theorem that for all $0<t<T$ and $u\in L^2(\mathbb R^n)$,
\begin{align*}
	\big\Vert\langle\xi_1,\nabla_x\rangle\ldots\langle\xi_m,\nabla_x\rangle e^{-t\mathcal P}u\big\Vert_{L^2}
	& \le \frac{e^{\frac m{2s}}c^m}{t^{k_{\xi_1}+\ldots+k_{\xi_m}+\frac m{2s}}}\ \bigg(\prod_{j=1}^m\vert\xi_j\vert\bigg)\ (m!)^{\frac 1{2s}}\ \big\Vert e^{-t\langle Bx,\nabla_x\rangle}u\big\Vert_{L^2} \\[5pt]
	& = \frac{e^{\frac m{2s}}c^me^{\frac12\Tr(B)t}}{t^{k_{\xi_1}+\ldots+k_{\xi_m}+\frac m{2s}}}\ \bigg(\prod_{j=1}^m\vert\xi_j\vert\bigg)\ (m!)^{\frac1{2s}}\ \Vert u\Vert_{L^2}.
\end{align*}
This ends the proof of Theorem \ref{09072019T1}.

\subsection{The nilpotent case} To end this section, we explain how the proof of Theorem \ref{09072019T1} can be adapted to obtain long-time asymptotics when the matrix $B$ is nilpotent (of order $K\geq1$ say), and to prove in particular the estimates stated in Proposition \ref{22062020P1}. In this case, we deduce from Propositions \ref{18072019P3} and \ref{18072019P6} in Appendix that there exists a positive constant $c_0>0$ such that
$$\forall t>0,\forall\xi\in\mathbb R^n,\quad\int_0^1\big\vert\sqrt Qe^{\alpha tB^T}\xi\big\vert^2\ \mathrm d\alpha\geq c_0\sum_{k=0}^{K-1}t^{2k}\big\vert\sqrt Q(B^T)^k\xi\big\vert^2,$$
and
$$\forall t>0,\quad\sup_{\xi\in\mathbb S^{n-1}\cap S^{\perp}}\bigg(\int_0^1\big\vert\sqrt Qe^{\alpha tB^T}\xi\big\vert^2\ \mathrm d\alpha\bigg)^{\frac12}\bigg(\int_0^1\big\vert\sqrt Qe^{\alpha tB^T}\xi\big\vert^{2s}\ \mathrm d\alpha\bigg)^{-\frac1{2s}}\le c_0.$$
From the same proof than the one of Proposition \ref{18072019P2}, we deduce that there exists another positive constant $c_1>1$ such that for all $k\in\{0,\ldots,K-1\}$, $q>0$, $t>0$, $\tau>0$ and $u\in L^2(\mathbb R^n)$,
\begin{equation}\label{22062020E6}
	\big\Vert\big\vert\sqrt Q(B^T)^kD_x\big\vert^qe^{-\tau a^w_t}u\big\Vert_{L^2}\le\bigg(\frac{c_1}{t^k\tau^{\frac1{2s}}}\bigg)^q\ q^{\frac q{2s}}\ \Vert u\Vert_{L^2}.
\end{equation}
We now assume that the fractional Ornstein-Uhlenbeck operator $\mathcal P$ is the fractional Kolmogorov operator $\mathcal K$ given by
$$\mathcal K = (-\Delta_v)^s + v\cdot\nabla_x,\quad(x,v)\in\mathbb R^{2n},$$
that is, that the matrices $B$ and $Q$ are respectively given by
$$B = \begin{pmatrix}
	0_n & I_n \\
	0_n & 0_n
\end{pmatrix}\quad \text{and}\quad Q = 2^{\frac1s}\begin{pmatrix}
	0_n & 0_n \\
	0_n & I_n
\end{pmatrix}.$$
In this particular case, the estimates \eqref{22062020E6} write in the following way for all $q>0$, $t>0$, $\tau>0$ and $u\in L^2(\mathbb R^n)$,
\begin{equation}\label{08072020E1}
	\big\Vert\big\vert D_x\big\vert^qe^{-\tau a^w_t}u\big\Vert_{L^2}\le\bigg(\frac{c_1}{t\tau^{\frac1{2s}}}\bigg)^q\ q^{\frac q{2s}}\ \Vert u\Vert_{L^2},\quad
	\big\Vert\big\vert D_v\big\vert^qe^{-\tau a^w_t}u\big\Vert_{L^2}\le\bigg(\frac{c_1}{\tau^{\frac1{2s}}}\bigg)^q\ q^{\frac q{2s}}\ \Vert u\Vert_{L^2}.
\end{equation}
Since the operators $\vert D_x\vert$, $\vert D_v\vert$ and $e^{-\frac t2 a^w_t}$ are Fourier multipliers, they commute, and by writing
$$\big\vert D_x\big\vert^{q_1}\big\vert D_v\big\vert^{q_2}e^{-ta^w_t} = \big\vert D_x\big\vert^{q_1}e^{-\frac t2a^w_t}\big\vert D_v\big\vert^{q_2}e^{-\frac t2a^w_t},$$
we therefore deduce from \eqref{18072019E1}, \eqref{18062020E3} and \eqref{08072020E1} that for all $q_1,q_2>0$, $t>0$ and $u\in L^2(\mathbb R^n)$,
$$\big\Vert\big\vert D_x\big\vert^{q_1}\big\vert D_v\big\vert^{q_2}e^{-t\mathcal K}u\big\Vert_{L^2}\le\bigg(\frac{c_1}{t^{1+\frac1{2s}}}\bigg)^{q_1}\ (2q_1)^{\frac{q_1}{2s}}\ \bigg(\frac{c_1}{t^{\frac1{2s}}}\bigg)^{q_2}\ (2q_2)^{\frac{q_2}{2s}}\ \Vert u\Vert_{L^2}.$$
These estimates imply in particular that there exists a positive constant $c_2>1$ such that $(\alpha,\beta)\in\mathbb N^{2n}$, $t>0$ and $u\in L^2(\mathbb R^{2n})$,
$$\big\Vert\partial^{\alpha}_x\partial^{\beta}_v(e^{-t\mathcal K}u)\big\Vert_{L^2}\le\frac{c_2^{\vert\alpha\vert+\vert\beta\vert}}{t^{(1+\frac1{2s})\vert\alpha\vert
+ \frac{\vert\beta\vert}{2s}}}\ (\alpha!)^{\frac1{2s}}\ (\beta!)^{\frac1{2s}}\ \Vert u\Vert_{L^2},$$
where we used the factorial estimate $m^m\le e^mm!$ holding for all $m\geq1$. This ends the proof of Proposition \ref{22062020P1}.

\section{Appendix}\label{Appendix}

Let $B$ and $Q$ be $n\times n$ real matrices, with $Q$ symmetric positive semidefinite. We consider the associated vector space $S$ defined in \eqref{032019E2} and $0\le r\le n-1$ the smallest integer such that \eqref{22072019E4} holds. In this appendix, we give the proofs of some results involving the matrices $B$ and $Q$.

\subsection{About the Kalman rank condition} First, we prove the characterization of the Kalman rank condition in term of the space $S$.

\begin{lem}\label{29082018E1} The Kalman rank condition \eqref{22072019E1} holds if and only if the vector space $S$ is reduced to $\{0\}$.
\end{lem}

\begin{proof} Using the notation of \eqref{22072019E1}, we have the following equivalences:
\begin{align*}
	\Rank\big[B\ \vert\ \sqrt Q\big] = n
	& \Leftrightarrow \Ran\big[B\ \vert\ \sqrt Q\big] = \mathbb R^n \\[5pt]
	& \Leftrightarrow \Ker\Big(\big[B\ \vert\ \sqrt Q\big]^T\Big) = \Big(\Ran\big[B\ \vert\ \sqrt Q\big]\Big)^{\perp} = \{0\} \\[5pt]
	& \Leftrightarrow S = \bigcap_{k=0}^{n-1}\Ker\big(\sqrt Q(B^T)^k\big) = \Ker\Big(\big[B\ \vert\ \sqrt Q\big]^T\Big) = \{0\},
\end{align*}
where $\perp$ denotes the orthogonality with respect to the canonical Euclidean structure. This ends the proof of Lemma \ref{29082018E1}.
\end{proof}

\subsection{A technical estimate} In this subsection, we give the proof of the estimate \eqref{09072019E2} which has a key role in the proof of Proposition \ref{18072019P2} in Section \ref{reg}. To that end, we will exploit the equivalence of norms in finite-dimensional normed vector spaces applied with Hardy's norms, as well as compactness arguments, inspired by techniques used in \cite{AB2} (Section 4).

\begin{prop}\label{18072019P3} In the general case, there exist some positive constants $c>0$ and $T>0$ such that for all $0\le t<T$ and $\xi\in\mathbb R^n$,
\begin{equation}\label{18072019E21}
	\int_0^1\big\vert\sqrt Qe^{\alpha tB^T}\xi\big\vert^2\ \mathrm d\alpha\geq c\sum_{k=0}^rt^{2k}\big\vert\sqrt Q(B^T)^k\xi\big\vert^2.
\end{equation}
In the particular case where the matrix $B$ is nilpotent of order $K\geq1$, there exists another positive constant $c_0>0$ such that for all $t>0$ and $\xi\in\mathbb R^n$,
\begin{equation}\label{08072020E2}
	\int_0^1\big\vert\sqrt Qe^{\alpha tB^T}\xi\big\vert^2\ \mathrm d\alpha\geq c_0\sum_{k=0}^{K-1}t^{2k}\big\vert\sqrt Q(B^T)^k\xi\big\vert^2.
\end{equation}
\end{prop}

\begin{proof} In order the simplify the notations, we set for all $t\geq0$ and $\xi\in\mathbb R^n$,
\begin{equation}\label{01042019E8}
	q_t(\xi) = \int_0^1\big\vert\sqrt Qe^{\alpha tB^T}\xi\big\vert^2\ \mathrm d\alpha = \big\Vert\sqrt Qe^{\alpha tB^T}\xi\big\Vert^2_{L^2(0,1)}.
\end{equation}
Notice that in the above $L^2(0,1)$ norm, the parameter $\alpha\in(0,1)$ stands for the variable of integration. This notation will be used throughout the proof each time a $L^2(0,1)$ norm appears. Moreover, the notation $\alpha$ will be used once as a polynomial variable for a Hardy's norm $\mathcal H^1$ on $(\mathbb R_r[\alpha])^n$, see \eqref{06092021E1} just after.

Minkowski's inequality implies that for all $t\geq0$ and $\xi\in\mathbb R^n$,
\begin{equation}\label{01042019E1}
	\sqrt{q_t(\xi)}\geq\bigg\Vert\sum_{k=0}^r\frac{(\alpha t)^k}{k!}\sqrt Q(B^T)^k\xi\bigg\Vert_{L^2(0,1)} - \bigg\Vert\sum_{k\geq r+1}\frac{(\alpha t)^k}{k!}\sqrt Q(B^T)^k\xi\bigg\Vert_{L^2(0,1)}.
\end{equation}
We focus on controlling the two terms appearing in the right-hand side of the above estimate. First, notice that on the finite-dimensional vector space $(\mathbb R_r[\alpha])^n$, the Hardy's norm $\Vert\cdot\Vert_{\mathcal H^1}$ defined by
\begin{equation}\label{06092021E1}
	\bigg\Vert\sum_{k=0}^ry_k\alpha^k\bigg\Vert_{\mathcal H^1} = \sum_{k=0}^rk!\vert y_k\vert,\quad y_1,\ldots,y_r\in\mathbb R^n,
\end{equation}
is equivalent to the standard Lebesgue's norm $\Vert\cdot\Vert_{L^2(0,1)}$ given by
$$\bigg\Vert\sum_{k=0}^ry_k\alpha^k\bigg\Vert^2_{L^2(0,1)} = \int_0^1\bigg\vert\sum_{k=0}^ry_k\alpha^k\bigg\vert^2\ \mathrm d\alpha,\quad y_1,\ldots,y_r\in\mathbb R^n.$$
This implies that there exists a positive constant $c_1>0$ such that for all $t\geq0$ and $\xi\in\mathbb R^n$,
\begin{equation}\label{01042019E2}
	\bigg\Vert\sum_{k=0}^r\frac{(\alpha t)^k}{k!}\sqrt Q(B^T)^k\xi\bigg\Vert_{L^2(0,1)}\geq c_1\sum_{k=0}^rt^k\big\vert\sqrt Q(B^T)^k\xi\big\vert,
\end{equation}
since 
$$\forall t\geq0, \forall\xi\in\mathbb R^n,\quad \sum_{k=0}^r\frac{(\alpha t)^k}{k!}\sqrt Q(B^T)^k\xi\in(\mathbb R_r[\alpha])^n.$$
In view of \eqref{18072019E21} and \eqref{01042019E1}, it remains to check that the remainder term 
$$\bigg\Vert\sum_{k\geq r+1}\frac{(\alpha t)^k}{k!}\sqrt Q(B^T)^k\xi\bigg\Vert_{L^2(0,1)},$$
can be controlled by $\sum_{k=0}^rt^k\big\vert\sqrt Q(B^T)^k\xi\big\vert$. Precisely, we will prove that there exist some positive constants $c_2>0$ and $t_1>0$ such that for all $0\le t<t_1$ and $\xi\in S^{\perp}$,
\begin{equation}\label{18072019E22}
	\bigg\Vert\sum_{k\geq r+1}\frac{(\alpha t)^k}{k!}\sqrt Q(B^T)^k\xi\bigg\Vert_{L^2(0,1)}\le c_2t\sum_{k=0}^rt^k\big\vert\sqrt Q(B^T)^k\xi\big\vert.
\end{equation}
Once the estimate \eqref{18072019E22} is established, we deduce from \eqref{01042019E1} and \eqref{01042019E2} that for all $0\le t<t_1$ and $\xi\in S^{\perp}$,
$$\sqrt{q_t(\xi)}\geq(c_1-c_2t)\sum_{k=0}^rt^k\big\vert\sqrt Q(B^T)^k\xi\big\vert.$$
This estimate combined with the triangular inequality implies that there exist other positive constants $c>0$ and $T>0$ such that for all $0\le t<T$ and $\xi\in S^{\perp}$,
\begin{equation}\label{18072019E24}
	q_t(\xi)\geq c\sum_{k=0}^rt^{2k}\big\vert\sqrt Q(B^T)^k\xi\big\vert^2.
\end{equation}
Moreover, for all $0\le t<T$, both quadratic forms $q_t$ and $\sum_{k=0}^rt^{2k}\vert\sqrt Q(B^T)^k\cdot\vert^2$ vanish on the vector space $S$ from \eqref{09092019E4}, which proves that the estimate \eqref{18072019E24} can be extended to all $0\le t<T$ and $\xi\in\mathbb R^n$ since $\mathbb R^n = S\oplus S^{\perp}$, according to the following elementary lemma whose proof is straightforward and omitted here

\begin{lem} Let $E$ be a real finite-dimensional vector space and $q_1,q_2$ be two non-negative quadratic forms on $E$. If $E=F\oplus G$ is a direct sum of two vector subspaces such that $q_1\le q_2$ on $F$ and $q_1,q_2$ both vanish on $G$, then $q_1\le q_2$ on $E$.
\end{lem}

\noindent We therefore need to check that the estimate \eqref{18072019E22} actually holds to end the proof of \eqref{18072019E21}. First, notice that there exists a positive constant $c_3>0$ such that for all $0\le t\le 1$ and $\xi\in\mathbb R^n$,
\begin{equation}\label{01042019E4}
	\bigg\Vert\sum_{k\geq r+1}\frac{(\alpha t)^k}{k!}\sqrt Q(B^T)^k\xi\bigg\Vert_{L^2(0,1)} = t^{r+1}\bigg\Vert\sum_{k\geq r+1}t^{k-r-1}\frac{\alpha^k}{k!}\sqrt Q(B^T)^k\xi\bigg\Vert_{L^2(0,1)}\le c_3t^{r+1}\vert\xi\vert.
\end{equation}
On the other hand, it follows from the definition \eqref{22072019E4} of the integer $0\le r\le n-1$ that there exists a positive constant $c_4>0$ such that for all $0\le t\le 1$ and $\xi\in S^{\perp}$,
\begin{equation}\label{01042019E5}
	\sum_{k=0}^rt^k\big\vert\sqrt Q(B^T)^k\xi\big\vert\geq t^r\sum_{k=0}^r\big\vert\sqrt Q(B^T)^k\xi\big\vert\geq c_4t^r\vert\xi\vert,
\end{equation}
the orthogonality being taken with respect to the canonical Euclidean structure of $\mathbb R^n$. Combining \eqref{01042019E4} and \eqref{01042019E5}, we obtain that for all $0\le t\le 1$ and $\xi\in S^{\perp}$,
$$\bigg\Vert\sum_{k\geq r+1}\frac{(\alpha t)^k}{k!}\sqrt Q(B^T)^k\xi\bigg\Vert_{L^2(0,1)}\le \frac{c_3t}{c_4}\sum_{k=0}^rt^k\big\vert\sqrt Q(B^T)^k\xi\big\vert.$$
This ends the proof of the estimate \eqref{18072019E22}.

When the matrix $B$ is nilpotent of order $K\geq1$, the matrix $B^T$ is also nilpotent with the same order and we deduce that 
$$\forall t\geq0,\forall\xi\in\mathbb R^n,\quad \sqrt Qe^{\alpha tB^T}\xi\in(\mathbb R_{K-1}[\alpha])^n.$$
The finite-dimensional argument of the beginning of this proof and the triangle inequality then allow to obtain the existence of a positive constant $c_0>0$ such that for all $t\geq0$ and $\xi\in\mathbb R^n$,
$$q_t(\xi)\geq c_0\sum_{k=0}^{K-1}t^{2k}\big\vert\sqrt Q(B^T)^k\xi\big\vert^2.$$
This ends the proof of the estimate \eqref{08072020E2}.
\end{proof}

\subsection{Study of the term $M_t$} To end this appendix, we establish the estimate \eqref{01052019E2}. The proof of the following proposition follows quite line to line the ones of \cite{AB} (Subsection 3.2) and is presented for the sake of completeness of the present work.

\begin{prop}\label{18072019P6} There exist some positive constants $c>0$ and $T>0$ such that for all $0<t<T$,
\begin{equation}\label{18072019E20}
	\sup_{\xi\in\mathbb S^{n-1}\cap S^{\perp}}\bigg(\int_0^1\big\vert\sqrt Qe^{\alpha tB^T}\xi\big\vert^2\ \mathrm d\alpha\bigg)^{\frac12}\bigg(\int_0^1\big\vert\sqrt Qe^{\alpha tB^T}\xi\big\vert^{2s}\ \mathrm d\alpha\bigg)^{-\frac1{2s}}\le c.
\end{equation}
When the matrix $B$ is nilpotent, the above estimate holds for all $t>0$.
\end{prop}

\begin{proof} Keeping the notation \eqref{17072019E1} from the proof of Proposition \ref{18072019P2}, we consider the term $M_t$ defined for all $t>0$ by
\begin{equation}
	M_t = \sup_{\xi\in\mathbb S^{n-1}\cap S^{\perp}}\bigg(\int_0^1\big\vert\sqrt Qe^{\alpha tB^T}\xi\big\vert^2\ \mathrm d\alpha\bigg)^{\frac12}\bigg(\int_0^1\big\vert\sqrt Qe^{\alpha tB^T}\xi\big\vert^{2s}\mathrm d\alpha\bigg)^{-\frac1{2s}}.
\end{equation}
Given $t>0$, let us first check that the term $M_t$ is well-defined. If the vector $\xi\in\mathbb S^{n-1}\cap S^{\perp}$ satisfies 
$$\int_0^1\big\vert\sqrt Qe^{\alpha tB^T}\xi\big\vert^p\ \mathrm d\alpha = 0,$$
for some $p\in\{2,2s\}$, we deduce that
$$\forall\alpha\in[0,1],\quad \sqrt Qe^{\alpha tB^T}\xi = 0.$$
By differentiating this identity with respect to $\alpha$ and evaluating in $\alpha = 0$, we deduce that
$$\forall k\in\{0,\ldots,r\},\quad \sqrt Q(B^T)^k\xi = 0.$$
We obtain from \eqref{22072019E4} that $\xi\in S$, which is not possible since $\xi\in\mathbb S^{n-1}\cap S^{\perp}$. We therefore proved that
$$\forall p\in\{2,2s\}, \forall\xi\in\mathbb S^{n-1}\cap S^{\perp},\quad \int_0^1\big\vert\sqrt Qe^{\alpha tB^T}\xi\big\vert^p\ \mathrm d\alpha>0.$$
Moreover, for all $p\in\{2,2s\}$, the functions 
$$\xi\in\mathbb S^{n-1}\cap S^{\perp}\mapsto\int_0^1\big\vert\sqrt Qe^{\alpha tB^T}\xi\big\vert^p\ \mathrm d\alpha,$$
are continuous on the compact set $\mathbb S^{n-1}\cap S^{\perp}$. This implies that the term $M_t$ is well defined and satisfies $0<M_t<+\infty$. 

In the remaining of this proof, we will widely use the following lemma whose proof is straightforward and can be found e.g. in \cite{AB} (Lemma 3.4). 

\begin{lem} \label{22062020L1} Let $E$ be a real finite-dimensional vector space and $L_1,L_2:E\rightarrow\mathbb{R}_+$ be two continuous functions satisfying for all $j\in\{1,2\}$,
$$\forall\lambda\geq 0, \forall P\in E,\quad L_j(\lambda P) = \lambda L_j(P),$$
and
$$\forall P\in E\setminus\{0\},\quad L_j(P) >0.$$
Then, there exists a positive constant $c>0$ such that
$$\forall P\in E,\quad L_1(P) \le cL_2(P).$$
\end{lem}

We now tackle the proof of the estimate \eqref{18072019E20}. For all $t>0$ and $\xi\in\mathbb S^{n-1}\cap S^{\perp}$, we consider 
$$M_t(\xi) = \bigg(\int_0^1\big\vert\sqrt Qe^{\alpha tB^T}\xi\big\vert^2\ \mathrm d\alpha\bigg)^{\frac12}\bigg(\int_0^1\big\vert\sqrt Qe^{\alpha tB^T}\xi\big\vert^{2s}\ \mathrm d\alpha\bigg)^{-\frac1{2s}}.$$
Let us first assume that the matrix $B$ is nilpotent of order $K\geq1$. Since the matrix $B^T$ is also nilpotent with the same index $K$, we have that
$$\forall t>0,\forall\xi\in\mathbb S^{n-1}\cap S^{\perp},\quad \sqrt Qe^{\alpha tB^T}\xi\in(\mathbb R_{K-1}[\alpha])^n.$$
It follows from Lemma \ref{22062020L1} applied with the real finite-dimensional vector space $E = (\mathbb R_{K-1}[\alpha])^n$ and the homogeneous continuous functions
\begin{equation}\label{22062020E4}
	L_1(P) = \bigg(\int_0^1\big\vert P(\alpha)\big\vert^2\ \mathrm d\alpha\bigg)^{\frac12}\quad\text{and}\quad L_2(P) = \bigg(\int_0^1\big\vert P(\alpha)\big\vert^{2s}\ \mathrm d\alpha\bigg)^{\frac1{2s}},
\end{equation}
that there exists a positive constant $c>0$ such that for all $t>0$ and $\xi\in\mathbb S^{n-1}\cap S^{\perp}$,
$$M_t(\xi)\le c.$$
This ends the proof of Proposition \ref{18072019P6} in this particular case.

Back to the case were the matrix $B$ is general, we consider $P_{t,\xi}$ and $R_{t,\xi}$ the functions defined for all $t>0$, $\xi\in\mathbb S^{n-1}\cap S^{\perp}$ and $\alpha\in[0,1]$ by
\begin{equation}\label{30072019E1}
	P_{t,\xi}(\alpha) = \sum_{k=0}^r\alpha^k\frac{t^k}{k!}\sqrt Q(B^T)^k\xi\quad\text{and}\quad R_{t,\xi}(\alpha) = \sqrt Qe^{t\alpha B^T}\xi - P_{t,\xi}(\alpha).
\end{equation}
It is fundamental in the following to notice that for all $t>0$ and $\xi\in\mathbb S^{n-1}\cap S^{\perp}$, the coordinates of the function $P_{t,\xi}$ are polynomials of degree less than or equal to $r$, that is
\begin{equation}\label{19072019E2}
	\forall t>0, \forall\xi\in\mathbb S^{n-1}\cap S^{\perp},\quad P_{t,\xi}\in (\mathbb R_r[\alpha])^n.
\end{equation}
By using \eqref{19072019E2} and Lemma \ref{22062020L1} with the real finite-dimensional vector space $E = (\mathbb R_r[\alpha])^n$ and the homogeneous continuous functions defined in \eqref{22062020E4} anew, we obtain the existence of a constant $c>0$ such that for all $t>0$ and $\xi\in\mathbb S^{n-1}\cap S^{\perp}$,
$$\bigg(\int_0^1\big\vert P_{t,\xi}(\alpha)\big\vert^2\ \mathrm d\alpha\bigg)^{\frac12}\le c\bigg(\int_0^1\big\vert P_{t,\xi}(\alpha)\big\vert^{2s}\ \mathrm d\alpha\bigg)^{\frac1{2s}}.$$
This estimate implies that for all $t>0$ and $\xi\in\mathbb S^{n-1}\cap S^{\perp}$,
\begin{equation}\label{estimationlapluslongdelunivers}
	M_t(\xi)
	\le c\left(\frac{\displaystyle\int_0^1\big\vert\sqrt Qe^{t\alpha B^T}\xi\big\vert^2\ \mathrm d\alpha}{\displaystyle\int_0^1\big\vert P_{t,\xi}(\alpha)\big\vert^2\ \mathrm d\alpha}\right)^{\frac12}\left(\frac{\displaystyle\int_0^1\big\vert P_{t,\xi}(\alpha)\big\vert^{2s}\ \mathrm d\alpha}{\displaystyle\int_0^1\big\vert\sqrt Qe^{\alpha tB^T}\xi\big\vert^{2s}\ \mathrm d\alpha}\right)^{\frac1{2s}}.
\end{equation}
We aim at establishing uniform upper bounds with respect to $t>0$ and $\xi\in\mathbb S^{n-1}\cap S^{\perp}$ for these two factors. To that end, we equip the vector space $(\mathbb R_r[\alpha])^n$ of the Hardy's norm $\Vert\cdot\Vert_{\mathcal H^\infty}$ defined by
$$\forall P\in(\mathbb R_r[\alpha])^n,\quad \Vert P\Vert_{\mathcal H^\infty} = \max_{k\in\{0,\ldots,r\}} \frac{\vert P^{(k)}(0)\vert}{k!}.$$
We deduce anew from \eqref{19072019E2} and Lemma \ref{22062020L1} applied with the real finite-dimensional vector space $E=(\mathbb R_r[\alpha])^n$ and the homogeneous continuous functions $\Vert\cdot\Vert_{\mathcal H^{\infty}}$ and \eqref{22062020E4} that
\begin{equation}\label{ineq_equiv}
	\forall p\in\{2,2s\}, \exists c_p>0, \forall P\in(\mathbb R_ r[\alpha])^n,\quad\Vert P\Vert_{\mathcal H^{\infty}}\le c_p\bigg(\int_0^1\vert P(\alpha)\vert^p\ \mathrm d\alpha\bigg)^{\frac1p}.  
\end{equation}
According to the definition \eqref{22072019E4} of the vector space $S$, we notice that
$$\forall\xi\in\mathbb S^{n-1}\cap S^{\perp},\quad \max_{k\in\{0,\ldots,r\}}\frac{\big\vert\sqrt Q(B^T)^k\xi\big\vert}{k!}>0.$$
Since the function
$$\xi\in\mathbb S^{n-1}\cap S^{\perp}\mapsto\max_{k\in\{0,\ldots,r\}}\frac{\big\vert\sqrt Q(B^T)^k\xi\big\vert}{k!},$$
is continuous on the compact set $\mathbb S^{n-1}\cap S^{\perp}$, we deduce that there exists a positive constant $\varepsilon>0$ such that
$$\forall\xi\in\mathbb S^{n-1}\cap S^{\perp},\quad\Vert P_{1,\xi}\Vert_{\mathcal H^{\infty}} = \max_{k\in\{0,\ldots,r\}}\frac{\big\vert\sqrt Q(B^T)^k\xi\big\vert}{k!}\geq\varepsilon.$$
It follows that
$$\forall t\in(0,1], \forall\xi\in\mathbb S^{n-1}\cap S^{\perp},\quad\Vert P_{t,\xi}\Vert_{\mathcal H^{\infty}} = \max_{k\in\{0,\ldots,r\}}\frac{t^k\big\vert\sqrt Q(B^T)^k\xi\big\vert}{k!}\geq\varepsilon t^r, $$
and we deduce from \eqref{ineq_equiv} that
\begin{equation}\label{ineq_Kalman}
	\forall p\in\{2,2s\},\forall t\in(0,1], \forall\xi\in\mathbb S^{n-1}\cap S^{\perp},\quad\varepsilon t^r\le c_p\bigg(\int_0^1\big\vert P_{t,\xi}(\alpha)\big\vert^p\ \mathrm d\alpha\bigg)^{\frac1p}.
\end{equation}
On the other hand, it follows from Taylor's formula with remainder term that
$$\forall t>0, \forall \xi\in\mathbb S^{n-1}\cap S^{\perp},\forall\alpha\in[0,1],\quad R_{t,\xi}(\alpha) = \frac{(t\alpha)^{r+1}}{r!}\int_0^1(1-\theta)^r\sqrt Q(B^T)^{r+1}e^{t\alpha\theta B^T}\xi \ \mathrm d\theta.$$
Therefore, there exists a positive constant $M>0$ such that
\begin{equation}\label{ineq_remainder}
	\forall t\in(0,1], \forall\xi\in\mathbb S^{n-1}\cap S^{\perp},\quad \Vert R_{t,\xi}\Vert_{L^{\infty}[0,1]} \le M t^{r+1}.
\end{equation}
With these estimates, we can obtain upper bounds on the two factors of the right-hand side of the estimate \eqref{estimationlapluslongdelunivers}. \\[5pt]
\textbf{1.} By using \eqref{30072019E1} and applying the triangle inequality for the $L^2$ norm, we first obtain
$$\forall t>0,\forall\xi\in\mathbb S^{n-1}\cap S^{\perp},\quad\left(\frac{\displaystyle\int_0^1\big\vert\sqrt Qe^{t\alpha B^T}\xi\big\vert^2\ \mathrm d\alpha}{\displaystyle\int_0^1\big\vert P_{t,\xi}(\alpha)\big\vert^2\ \mathrm d\alpha}\right)^{\frac12} 
\le 1+\left(\frac{\displaystyle\int_0^1\big\vert R_{t,\xi}(\alpha)\big\vert^2\ \mathrm d\alpha}{\displaystyle\int_0^1\big\vert P_{t,\xi}(\alpha)\big\vert^2\ \mathrm d\alpha}\right)^{\frac12}.$$
According to \eqref{ineq_Kalman} and \eqref{ineq_remainder}, we therefore get that for all $0<t\le1$ and $\xi\in\mathbb S^{n-1}\cap S^{\perp}$,
\begin{equation}\label{22052018E1}
	\left(\frac{\displaystyle\int_0^1\big\vert\sqrt Qe^{t\alpha B^T}\xi\big\vert^2\ \mathrm d\alpha}{\displaystyle\int_0^1\big\vert P_{t,\xi}(\alpha)\big\vert^2\ \mathrm d\alpha}\right)^{\frac12} 
	\le 1 + \frac{c_2Mt^{r+1}}{\varepsilon t^r} = 1 + \frac{c_2 M}{\varepsilon}t
	\le 1 + \frac{c_2 M}{\varepsilon}.
\end{equation}
\textbf{2.} Notice that the classical estimate
$$\forall\xi,\eta\in\mathbb R^n,\quad \vert\xi+\eta\vert^{2s}\le 2^{(2s-1)_+}(\vert\xi\vert^{2s}+\vert\eta\vert^{2s}),$$
with $(2s-1)_+ = \max(2s-1,0)$, implies the following one
$$\forall\xi,\eta\in\mathbb R^n,\quad 2^{-(2s-1)_+}\vert\xi\vert^{2s}-\vert\eta\vert^{2s}\le\vert\xi-\eta\vert^{2s}.$$
We therefore deduce from \eqref{30072019E1} that for all $0<t\le 1$ and $\xi\in\mathbb S^{n-1}\cap S^{\perp}$,
$$\frac{\displaystyle\int_0^1\big\vert\sqrt Qe^{t\alpha B^T}\xi\big\vert^{2s}\ \mathrm d\alpha}{\displaystyle\int_0^1\big\vert P_{t,\xi}(\alpha)\big\vert^{2s}\ \mathrm d\alpha}
\geq 2^{-(2s-1)_+} - \frac{\displaystyle\int_0^1\big\vert R_{t,\xi}(\alpha)\big\vert^{2s}\ \mathrm d\alpha}{\displaystyle\int_0^1\big\vert P_{t,\xi}(\alpha)\big\vert^{2s}\ \mathrm d\alpha}.$$
Moreover, it follows from \eqref{ineq_Kalman} and \eqref{ineq_remainder} that for all $0<t\le1$ and $\xi\in\mathbb S^{n-1}\cap S^{\perp}$,
$$\frac{\displaystyle\int_0^1\big\vert R_{t,\xi}(\alpha)\big\vert^{2s}\ \mathrm d\alpha}{\displaystyle\int_0^1\big\vert P_{t,\xi}(\alpha)\big\vert^{2s}\ \mathrm d\alpha}
\le\bigg(\frac{c_{2s}Mt^{r+1}}{\varepsilon t^r}\bigg)^{2s} = \bigg(\frac{c_{2s}M}{\varepsilon}t\bigg)^{2s},$$
from which we deduce that
$$\frac{\displaystyle\int_0^1\big\vert\sqrt Qe^{t\alpha B^T}\xi\big\vert^{2s}\ \mathrm d\alpha}{\displaystyle\int_0^1\big\vert P_{t,\xi}(\alpha)\big\vert^{2s}\ \mathrm d\alpha}
\geq 2^{-(2s-1)_+} - \bigg(\frac{c_{2s}M}{\varepsilon}t\bigg)^{2s}.$$
It follows that there exist some positive constants $c_0>0$ and $0<t_0<1$ such that for all $0<t<t_0$ and $\xi\in\mathbb S^{n-1}\cap S^{\perp}$,
\begin{equation}\label{23092018E1}
	\left(\frac{\displaystyle\int_0^1\big\vert\sqrt Qe^{t\alpha B^T}\xi\big\vert^{2s}\ \mathrm d\alpha}{\displaystyle\int_0^1\big\vert P_{t,\xi}(\alpha)\big\vert^{2s}\ \mathrm d\alpha}\right)^{\frac1{2s}}\geq c_0.
\end{equation}
As a consequence of \eqref{estimationlapluslongdelunivers}, \eqref{22052018E1} and \eqref{23092018E1}, there exists a positive constant $c_1>0$ such that
$$\forall t\in(0,t_0), \forall \xi\in\mathbb S^{n-1}\cap S^{\perp},\quad M_t(\xi)\le c_1.$$
This ends the proof of the estimate \eqref{18072019E20} in the general case.
\end{proof}

\end{document}